\documentclass{article}

\usepackage{a4wide}
\usepackage{amsfonts}
\usepackage{amsmath}
\usepackage{amsthm}
\usepackage[english]{babel}
\usepackage{bbm}

\newcommand{\N}{\mathbb N}

\newcommand{\R}{\mathbb R}

\newcommand{\half}{\mbox{$\frac 1 2$}}

\newtheorem{theorem}{Theorem}[section]
\newtheorem{lemma}[theorem]{Lemma}
\newtheorem{proposition}[theorem]{Proposition}

\newtheorem{corollary}[theorem]{Corollary}

\theoremstyle{definition}
\newtheorem{definition}[theorem]{Definition}
\newtheorem{hypothesis}[theorem]{Hypothesis}

\theoremstyle{remark}
\newtheorem{remark}[theorem]{Remark}

\newcommand{\D}{\mathfrak D}

\title{Dissipativity of the delay semigroup}
\author{Joris Bierkens\footnote{Joris Bierkens, Donders Institute for Brain, Cognition and Behaviour, Radboud University Nijmegen, j.bierkens@science.ru.nl} \quad Onno van Gaans\footnote{Onno van Gaans, Mathematical Institute, Universiteit Leiden, vangaans@math.leidenuniv.nl}}
\date{}

\begin{document}

\maketitle



\begin{abstract}
Under mild conditions a delay semigroup can be transformed into a (generalized) contraction semigroup by modifying the inner product on the (Hilbert) state space into an equivalent inner product. Applications to stability of differential equations with delay and stochastic differential equations with delay are given as examples.
\end{abstract}

\section{Introduction}
Consider a linear functional evolution equation of the form
\begin{equation} \label{eq:linearfunctional} \left\{ \begin{array}{ll} \frac{d}{dt} u(t) = B u(t) + \int_{-r}^0 d \zeta(\sigma) \ u(t+\sigma), \quad & t \geq 0, \\ u(0) = x \\ u(\sigma) = f(\sigma), \quad -1 \leq \sigma \leq 0 \end{array} \right. \end{equation}
in a Hilbert space $X$ where $\zeta : [-1,0] \rightarrow L(X)$ of bounded variation, with initial condition $x \in X$ and $f \in L^2([-1,0];X)$.
Under conditions on $B$ and $\zeta$ there exists a unique solution to~\eqref{eq:linearfunctional}, see \cite{batkai}.
If we let $\mathcal E^2 = X \times L^2([-1,0];X)$ denote the state space of this functional evolution equation then we may consider the solution semigroup of this linear evolution:
\begin{equation} \label{eq:delaysemigroup} T(t) \begin{pmatrix} x \\ f \end{pmatrix} = \begin{pmatrix} u(t) \\ u(t + \cdot) \end{pmatrix}.\end{equation}
We call $(T(t))_{t \geq 0}$ an abstract delay semigroup. This notion is introduced in more detail in Section~\ref{sec:delaysemigroup}.

The semigroup thus obtained is rarely a contraction semigroup. In this paper we will show that under very mild conditions, we may change the inner product on $\mathcal E^2$ into an equivalent inner product, such that the delay semigroup becomes a (generalized) contraction semigroup. A \emph{contraction semigroup} is a semigroup $(S(t))_{t \geq 0}$ such that $||S(t)|| \leq 1$ for all $t \geq 0$. A \emph{generalized contraction semigroup} is a semigroup $(S(t))_{t \geq 0}$ such that $||S(t)|| \leq e^{\lambda t}$ for some $\lambda \in \R$ and all $t \geq 0$.

As explained in Section~\ref{sec:stability_delay_eq} our results improve upon \cite{webb1976}, where dissipativity of functional differential equations is established under stronger conditions, and upon \cite[Section 10.3]{dapratozabczyk1996}, where the results are restricted to linear delay equations of the form
\[ \frac{d}{dt} x(t) = B x(t) + \sum_{i=1}^k B_i x(t-h_i),\]
and where the conditions are hard to check in practice.

In Section~\ref{sec:delaysemigroup} we briefly describe the construction of the abstract delay semigroup, mainly to introduce notation. Then our main results are stated in Section~\ref{sec:dissipativity}. We sketch some possible applications, mainly in the field of stochastic evolution equations, in Section~\ref{sec:applications}.

\section{Abstract delay differential equations}
\label{sec:delaysemigroup}
We present here the abstract framework for the study of deterministic delay differential equations, or \emph{delay equation}, of~\cite{batkai}. Deterministic delay equations may also be studied in spaces of continuous functions, see~\cite{diekmann1995} and \cite[Section VI.6]{engelnagel}, but this setting is not discussed here.

Let $(X,|\cdot|)$ be a Banach space, and for $1 \leq p < \infty$ let $W^{1,p}([-1,0];X)$ denote the \emph{Sobolev space} consisting of equivalence classes of functions in $L^p([-1,0];X)$ which have a weak derivative in $L^p([-1,0];X)$ (see \cite{taylor1996}, Chapter 4).

Consider the abstract differential equation with delay
\begin{equation}\label{eq:abstractdelayeq} \left\{ \begin{array}{l}\frac{d}{dt} u(t) = B u(t) + \Phi u_t, \quad t > 0, \\
 u(0) = x, \\
 u_0 = f, \end{array} \right.
\end{equation}
under the following assumptions:

\begin{hypothesis}
\label{hyp:defabstractdelay}
\begin{itemize}
\item[(i)] $x \in X$;
\item[(ii)] $B$ is the generator of a strongly continuous semigroup $(S(t))_{t \geq 0}$ in $X$;
\item[(iii)] $f \in L^p([-1,0];X)$, $1 \leq p < \infty$;
\item[(iv)] $\Phi : W^{1,p}([-1,0];X) \rightarrow X$ is a bounded linear operator, allowing the expression 
\[ \Phi(f) := \int_{-1}^0 d \zeta(s) \ f(s),\] 
for $f \in C([-1,0];X)$, where $\zeta : [-1,0] \rightarrow L(X)$ is of bounded variation, and where the integral is a Lebesgue-Stieltjes integral.
\item[(v)] $u: [-1,\infty) \rightarrow X$ and for $t \geq 0$, $u_t : [-1,0] \rightarrow X$ is defined by $u_t(\sigma) = u(t+\sigma)$, $\sigma \in [-1,0]$.
\end{itemize}
\end{hypothesis}
In general, if $(\xi(t))_{t \in [-1,\infty)}$ is a process, then the process $(\xi_t)_{t \geq 0}$ with values in a function space, defined by $\xi_t(\sigma) := \xi(t+\sigma)$, $t \geq 0$, $\sigma \in [-1,0]$, is called the \emph{segment process} of $\xi$. So here $(u_t)_{t \geq 0}$ is the segment process of $(u(t))_{t \in [-1,\infty)}$. It keeps track of the history of $(u(t))_{t \in [-1,\infty)}$.

\begin{definition}
A \emph{classical solution}\index{classical solution of abstract delay differential equation} of~(\ref{eq:abstractdelayeq}) is a function $u: [-1,\infty) \rightarrow X$ that satisfies
\begin{itemize}
\item[(i)] $u \in C([-1,\infty);X) \cap C^1([0,\infty);X)$;
\item[(ii)] $u(t) \in \D(B)$ and $u_t \in W^{1,p}([-1,0];X)$ for all $t \geq 0$;
\item[(iii)] $u$ satisfies~(\ref{eq:abstractdelayeq}).
\end{itemize}
\end{definition}

To employ a semigroup approach we introduce the Banach space 
\[ \mathcal E^p := X \times L^p([-1,0];X),\]
and the closed, densely defined operator in $\mathcal E^p$,
\begin{equation} \label{eq:generatorabstractdelaysemigroup} A := \begin{bmatrix} B & \Phi \\ 0 & \frac{d}{d \sigma} \end{bmatrix}, \quad \D(A) = \left\{ \begin{pmatrix} x \\ f \end{pmatrix} \in \D(B) \times W^{1,p}([-1,0];X) : f(0) = x \right\}.\end{equation}

The equation~(\ref{eq:abstractdelayeq}) is called \emph{well-posed}\index{well-posed delay differential equation} if for all $(x,f) \in \D(A)$, there exists a unique classical solution of~(\ref{eq:abstractdelayeq}) that depends continuously on the initial data (in the sense of uniform convergence on compact intervals).

If Hypothesis~\ref{hyp:defabstractdelay} is satisfied, then $A$ generates a strongly continuous semigroup \cite[Theorem 3.26, 3.29]{batkai}. By~\cite[Corollary 3.7]{batkai} this is equivalent to saying that~\eqref{eq:abstractdelayeq} is well-posed. In this case the semigroup generated by $A$ is called an \emph{(abstract) delay semigroup}.

\section{Dissipativity of the delay semigroup}
\label{sec:dissipativity}

We continue to use the notation of the previous section, under the assumptions of Hypothesis~\ref{hyp:defabstractdelay}; however, we will specialize to the situation in which $(X, \langle \cdot, \cdot \rangle)$ is a Hilbert space and $p=2$.
First recall the notions of dissipativity and contraction semigroup on Hilbert spaces.

\begin{definition}
Suppose $A : \D(A) \rightarrow H$ is a linear operator with $\D(A) \subset H$, and where $(H, \langle \cdot,\cdot\rangle)$ is a Hilbert space. Then $A$ is said to be \emph{dissipative} if $\langle A x, x \rangle \leq 0$ for all $x \in \D(A)$.
A strongly continuous semigroup $(T(t))_{t \geq 0}$ on a Banach space $Y$ is said to be a \emph{contraction semigroup} if $||T(t))|| \leq 1$ for all $t \geq 0$.
\end{definition}

It is well-known that a strongly continuous semigroup is a contraction semigroup if and only if the infinitesemal generator of the semigroup is dissipative \cite[Section II.3.a]{engelnagel}. In the following, if $\zeta : [-1,0] \rightarrow L(X)$ is of bounded variation, let $|\zeta| : [-1,0] \rightarrow \R$ denote the total variation process, 
\[ |\zeta|(t) = \sup \sum_{(t_i,t_{i+1}) \in \pi } \left|||\zeta||(t_{i+1}) - ||\zeta||(t_i)\right|,\]
where the supremum ranges over all partitions $\pi = \{ t_0, \dots, t_n\}$ of the interval $[0,t]$ of the form $0 = t_0 < t_1 < \dots < t_n = t$.

\begin{theorem}
\label{thm:dissipativityboundedvariation}
Assume Hypothesis~\ref{hyp:defabstractdelay} holds.
Suppose furthermore that $B - \lambda I$ is dissipative.
If 
\begin{equation} \label{eq:conditiondissipativity} (\mu - \lambda)^2 > (|\zeta|(0) - |\zeta|(-1))\int_{-1}^0 e^{2 \mu r} \ d |\zeta|(r)\end{equation}
for some $\mu > \lambda$, 
then there exists an equivalent inner product on $X \times L^2([-1,0];X)$ (with respect to the canonical inner product on $X \times L^2([-1,0];X)$) such that $A - \mu I$ is dissipative with respect to this inner product. This inner product is given by
\begin{equation} \label{eq:inner_product} \left(\begin{pmatrix} c \\ f \end{pmatrix}, \begin{pmatrix} d \\ g \end{pmatrix} \right) := \langle c, d \rangle + \int_{-1}^0 \tau(s) \langle f(s), g(s) \rangle \ d s, \quad c, d \in X, f, g \in L^2([-1,0];X),\end{equation}
where $\tau : [-1,0] \rightarrow \R$ is given by 
\begin{equation}
\label{eq:tau}
\tau(t) = e^{-2 \mu t} \left[ \mu - \lambda - \frac{|\zeta|(0) - |\zeta|(-1)}{\mu - \lambda} \int_t^0 e^{2 \mu r} \ d |\zeta|(r) \right].
\end{equation}
\end{theorem}

In the proof we will make use of the following two lemmas.

\begin{lemma}
\label{lem:tau_positive}
Let $\tau: [-1,0] \rightarrow \R$ be defined by~\eqref{eq:tau} and suppose that $\mu > \lambda$ and ~\eqref{eq:conditiondissipativity} holds. Then 
\[  0 < \inf_{t \in [-1,0]} \tau(t) \leq \sup_{t \in [-1,0]} \tau(t) < \infty.\]
\end{lemma}

\begin{proof}
Clearly $\tau$ is bounded from above. The lower bound follows by noting that $\inf_{t \in [-1,0]} e^{-2 \mu t} > 0$ and, under condition~\eqref{eq:conditiondissipativity},
\begin{align*}
\sup_{t \in [-1,0]}\frac{|\zeta|(0) - |\zeta|(-1)}{\mu - \lambda} \int_t^0 e^{2 \mu r} \ d |\zeta|(r) & \leq  \frac{|\zeta|(0) - |\zeta|(-1)}{\mu - \lambda} \int_{-1}^0 e^{2 \mu r} \ d |\zeta|(r) < \mu - \lambda.
\end{align*}

\end{proof}

\begin{lemma}
\label{lem:integration_by_parts}
Let $\tau : [-1,0] \rightarrow \R$ be defined by~\eqref{eq:tau}. Then for $f \in W^{1,2}([-1,0];\R)$ we have
\begin{equation}
\int_{-1}^0 \tau(s) \left( \half \frac{d}{ds} f(s) - \mu f(s) \right) \ d s = \half (\mu - \lambda) f(0) - \half \tau(-1) f(-1) -  \frac{ |\zeta|(0) - |\zeta|(-1)}{2 (\mu - \lambda)} \int_{-1}^0 f(s) \ d |\zeta|(s).
\end{equation}

\end{lemma}
\begin{proof}
Let $\tau_1(t) = e^{-2 \mu t}$ and $\tau_2(t) = e^{-2 \mu t} \int_t^0 e^{2 \mu r} d |\zeta|(r)$, $c_1 = \mu - \lambda$ and $c_2 = -\frac{ |\zeta|(0) - |\zeta|(-1)}{\mu - \lambda}$, so that 
\[ \tau(t) = c_1 \tau_1(t) + c_2 \tau_2(t).\]
Using integration by parts we compute
\begin{align*}
 \int_{-1}^0 \tau_1(s) \frac{d}{ds} f(s) \ d s & = \tau_1(0) f(0) - \tau_1(-1) f(-1) - \int_{-1}^0 f(s) \ d \tau_1(s) \\
 & = f(0) - \tau_1(-1) f(-1) + 2 \mu \int_{-1}^0 f(s) e^{-2 \mu s} \ d s,
\end{align*}
and
\begin{align*}
&  \int_{-1}^0 \tau_2(s) \frac{d}{ds} f(s) \ d s  \\
& = \tau_2(0) f(0) - \tau_2(-1) f(-1) - \int_{-1}^0 f(s) \ d \tau_2(s) \\
 & = - \tau_2(-1) f(-1) + 2 \mu  \int_{-1}^0 \left\{ f(s) e^{-2 \mu s} \int_s^0 e^{2 \mu r} \ d |\zeta|(r) \right\} \ d s + \int_{-1}^0 f(s) \ d |\zeta|(s) \\
 & = - \tau_2(-1) f(-1) + 2 \mu  \int_{-1}^0  f(s) \tau_2(s) \ d s 
 + \int_{-1}^0 f(s) \ d |\zeta|(s) 
\end{align*}
The claimed result follows by combining these expressions.
\end{proof}

\begin{proof}[Proof of Theorem~\ref{thm:dissipativityboundedvariation}]
Let $\left( \cdot, \cdot \right)$ denote the symmetric bilinear form on $X \times L^2([-1,0];X)$ defined by~\eqref{eq:inner_product}. By the lower bound for $\tau$ in Lemma~\ref{lem:tau_positive}, it is in fact a coercive bilinear form so that it defines an inner product on $X \times L^2([-1,0];X)$. Also, using the finiteness of the upper bound as stated in Lemma~\ref{lem:tau_positive}, 
\[ \left( \begin{pmatrix} c \\ f \end{pmatrix},  \begin{pmatrix} c \\ f \end{pmatrix} \right) = |c|^2 + \int_{-1}^0 \tau(s) |f(s)|^2 \ d s \leq \left( 1 \vee \sup_{t \in [-1,0]} \tau(t) \right) \left( |c|^2 + \int_{-1}^0 |f(s)|^2 \ d s \right),\]
so that the inner product $(\cdot,\cdot)$ is equivalent to the canonical inner product on $X \times L^2([-1,0];X)$.

We have for $x \in \D(A)$,
\begin{align*}
( (A - \mu I) x, x )& = \langle B x(0),x(0)\rangle  + \left\langle \int_{-1}^0 d \zeta(s) x(s), x(0) \right\rangle \\
& \quad + \int_{-1}^0 \tau(s) \langle \dot x(s), x(s) \rangle \ d s - \mu |x(0)|^2 - \mu \int_{-1}^0 \tau(s) |x(s)|^2 \ d s \\
& \leq (\lambda-\mu) |x(0)|^2 + \int_{-1}^0 |x(0)| |x(s)| \ d |\zeta|(s) + \int_{-1}^0 \tau(s) \left( \half \frac{d}{ds } |x(s)|^2 - \mu |x(s)|^2 \right) \ d s
\end{align*}

By Lemma~\ref{lem:integration_by_parts}, applied to $f(\cdot) = |x(\cdot)|^2$, we therefore have
\begin{align*}
 ( (A - \mu I) x, x ) & \leq  (\lambda-\mu) |x(0)|^2 + \int_{-1}^0 |x(0)| |x(s)| \ d |\zeta|(s) \\
 & \quad \quad + \half (\mu - \lambda) |x(0)|^2 - \half \tau(-1) |x(-1)|^2 -  \frac{ |\zeta|(0) - |\zeta|(-1)}{2 (\mu - \lambda)} \int_{-1}^0 |x(s)|^2 \ d |\zeta|(s) \\
 & \leq - \half (\mu - \lambda) |x(0)|^2 + \int_{-1}^0 |x(0)| |x(s)| \ d |\zeta|(s) -  \frac{ |\zeta|(0) - |\zeta|(-1)}{2 (\mu - \lambda)} \int_{-1}^0 |x(s)|^2 \ d |\zeta|(s) \\
 & = |x(0)|^2  \int_{-1}^0 \left\{ - \frac{\mu - \lambda}{2 (|\zeta|(0) -|\zeta|(-1))} + \frac{|x(s)|}{|x(0)|} - \frac{|\zeta|(0) - |\zeta|(-1)}{2 (\mu - \lambda)} \frac{|x(s)|^2}{|x(0)|^2}\right\} \ d |\zeta|(s),
\end{align*}
where we used that $\tau(-1) > 0$ in the second inequality, and assumed $|\zeta|(0) > |\zeta|(-1)$. (If $|\zeta|(0) = |\zeta|(-1)$, then $\zeta$ is constant, and the expression after the second inequality above is then less than or equal to zero, which we want to show.)
Note that the integrand is a second order polynomial in $\frac{|x(s)|}{|x(0)|}$, of the form
\[ p(\xi) = - \frac 1 {2 \alpha} + \xi - \half \alpha \xi^2 = -\half \alpha \left( \xi - \frac 1 {\alpha} \right)^2 \leq 0, \]
since $\alpha = \frac{|\zeta|(0) - |\zeta|(-1)}{\mu - \lambda} \geq 0$. As a result, it follows that $((A - \mu I)x ,x ) \leq 0$ for all $x \in \D(A)$.

\end{proof}

It is in general not possible to verify condition~\eqref{eq:conditiondissipativity} explicitly. However, we may deduce the following simpler conditions.

\begin{corollary}
\label{cor:dissipativitydelay}
Assume Hypothesis~\ref{hyp:defabstractdelay} holds.
Suppose furthermore that $B - \lambda I$ is dissipative.
\begin{itemize}
\item[(i)]
If 
\begin{equation} \label{eq:conditionmu} \mu > \lambda + \left( (|\zeta|(0) - |\zeta|(-1))\int_{-1}^0 e^{2 \lambda r} \ d |\zeta|(r) \right)^{1/2}\end{equation}
for some $\mu \in \R$, 
then there exists an equivalent inner product on $\mathcal E^2$ such that $A - \mu I$ is dissipative.
\item[(ii)]
If 
\[ \lambda < -\int_{-1}^0 d |\zeta|,\]
then there exists an equivalent inner product on $\mathcal E^2$ such that
$A$ is dissipative.
\end{itemize}
\end{corollary}

\begin{proof}
The first statement follows from $e^{2 \mu r} < e^{2 \lambda r}$ for $\lambda < \mu$ and all $r \leq 0$, and rewriting~\eqref{eq:conditiondissipativity}. The second statement follows by noting that $\lambda ^2 > (|\zeta|(0) - |\zeta|(-1))^2$ so that~\eqref{eq:conditiondissipativity} is satisfied with $\mu = 0$.
\end{proof}

To state an important consequence, we recall the notion of generalized contraction.
\begin{definition}
A strongly continuous semigroup $(T(t))_{t \geq 0}$ on a Banach space $(Y, ||\cdot||)$ is called a \emph{generalized contraction} if there exists a constant $\mu \in \R$ such that $||T(t)||_{t \geq 0} \leq e^{\mu t}$.
\end{definition}

\begin{theorem}
\label{thm:generalizedcontractiongivesgeneralizedcontraction}
Suppose Hypothesis~\ref{hyp:defabstractdelay} holds. Suppose furthermore that $B$ is the generator of a generalized contraction semigroup. 
Then there exists an equivalent inner product on $\mathcal E^2$ such that $A$ is the generator of a generalized contraction semigroup.
\end{theorem}

\begin{proof}
Denote the semigroup generated by $B$ by $(S(t))_{t \geq 0}$ and suppose that 
\[ ||S(t)|| \leq e^{\lambda t}, \quad t \geq 0.\]
Let $\nu > \max\left(0, \lambda + \int_{-1}^0 \ d |\zeta| \right)$.
Define $\widetilde B := B - \nu I$ and $\widetilde \lambda := \lambda - \nu$.
It may be verified that the conditions of Corollary~\ref{cor:dissipativitydelay} are satisfied for $\widetilde B$ and $\widetilde \lambda$, so that an equivalent inner product $(\cdot,\cdot)$ on $\mathcal E^2$ exists such that the delay semigroup generated by 
\[ \widetilde A := \begin{bmatrix} B - \nu I & \Phi \\ 0 & \frac{d}{d\sigma} \end{bmatrix} \]
is dissipative.
Now for $x \in \D(A)$
\[ ((A - \nu I) x,x) = (\widetilde A x, x) - \nu \int_{-1}^0 \tau(\sigma) x(\sigma)^2 \ d \sigma \leq (\widetilde A x, x) \leq 0. \]
\end{proof}

Note furthermore that, if $A$ is of the form~(\ref{eq:generatorabstractdelaysemigroup}) with $B - \lambda I$ dissipative for some $\lambda \in \R$, we can always perturb $A$ by a bounded operator of the form $\begin{pmatrix} - c I & 0 \\ 0 & 0 \end{pmatrix}$ to obtain the generator of a new delay semigroup. If we choose $c > 0$ large enough, by Corollary~\ref{cor:dissipativitydelay} we may change the inner product to obtain a dissipative generator.

In an entirely analogous way as for Theorem~\ref{thm:dissipativityboundedvariation} we can prove the following slightly stronger result in case $\Phi$ has a density function.

\begin{proposition}
Suppose $\Phi$ is of the form
\[ \Phi f = \int_{-1}^0 \zeta(\sigma) f(\sigma) \ d \sigma, \quad f \in L^2([-1,0];X),\]
with $\zeta \in L^2([-1,0];L(X))$.
Suppose furthermore $B - \lambda I$ is dissipative.
If there exists $\mu > \lambda$ such that
\[(\lambda - \mu)^2 >  \int_{-1}^0 e^{2 \mu \rho} | \zeta(\rho)|^2 \ d \rho.\]
then there exists an equivalent inner product on $\mathcal E_2$ such that $A -\mu I$ is dissipative with respect to this inner product.
\end{proposition}

\subsection{Abstract sufficient conditions for dissipativeness}
The following is an attempt at establishing conditions on more general semigroup generators $A$ such that there exists an equivalent inner product with respect to which $A$ is dissipative. The remarks in this short section play no role in the remainder of this paper. 
We make use of the notion of exact observability.

\begin{definition}
Let $A$ be the generator of a strongly continuous semigroup $(T(t))_{t \geq 0}$ in a Hilbert space $X$. Let $Y$ be a Hilbert space and $C \in L(X;Y)$. Then $(A,C)$ is said to be \emph{exactly observable} in time $\tau > 0$ if the mapping $x \mapsto C^{\tau} x := C T(\cdot) x : X \rightarrow L^2([0,\tau];Y)$ is injective and its inverse is bounded on the range of $C^{\tau}$.
\end{definition}

Observability is a concept dual to controllability: $(A,C)$ is exactly observable if and only if $(A, C^*)$ is exactly controllable (see \cite{curtainzwart}, Section 4.1).

\begin{proposition}
Suppose $A$ generates an asymptotically stable strongly continuous semigroup $(T(t))_{t \geq 0}$ in $X$. If there exists a Hilbert space $Y$, a bounded linear mapping $C \in L(X;Y)$ and $\tau > 0$ such that $(A,C)$ is exactly observable in time $\tau$, then there exists an equivalent inner product on $X$ such that $A$ is dissipative with respect to this inner product.
\end{proposition}

\begin{proof}
For this $C$, by the characterization of exactly observable systems (\cite{curtainzwart}, Corollary 4.1.14) we have
\[ Q := \int_0^{\infty} T^*(t) C^* C T(t) \ d t \geq \int_0^\tau T^*(t) C^* C T(t) \ d t \geq \gamma I  \]
for some $\gamma > 0$.
Furthermore by Lyapunov theory (\cite{curtainzwart}, Theorem 4.1.23) we have that $Q \in L(X)$ and 
\[ 2 \langle QA x, x \rangle  = - |C x|^2 \leq 0, \quad x \in \D(A).\]
\end{proof}

\section{Applications}
\label{sec:applications}

In this section we provide some examples of possible applications of the dissipativity property that we established in the previous section.

\subsection{Stability of delay differential equations}
\label{sec:stability_delay_eq}
The conditions of Theorem~\ref{thm:dissipativityboundedvariation} and Corollary~\ref{cor:dissipativitydelay} give us stability results for the delay semigroup. Condition (ii) of Corollary~\ref{cor:dissipativitydelay} is `sharp' in the sense that it reproduces the stability result \cite{batkai}, Corollary 5.9.

For the linear delay equation in $\R^n$ consisting of two terms,
\[ \dot x(t) = B x(t) + C x(t-1), \]
with $B, C \in \R^{n \times n}$, and $B - \lambda I$ dissipative, Corollary~\ref{cor:dissipativitydelay} gives the results
\begin{itemize}
\item[(i)] If $\mu > \lambda + ||C|| e^{-\lambda}$ then $A - \mu I$ is dissipative under an equivalent inner product, 
\item[(ii)] If $\lambda + ||C|| < 0$ then $A$ is dissipative under an equivalent inner product.
\end{itemize}
Result (i) generalizes \cite[Remark 10.3.2]{dapratozabczyk1996} to the multidimensional case.
It may be compared to the result by G.F. Webb \cite[Proposition 4.1]{webb1976} which states that $A - \mu I$ is dissipative if $\mu \geq \max(0,\lambda+||C||)$. In case $\lambda > 0$ or $\lambda e^{\lambda} < -||C||$ our result (i) is stronger. 

Consider now for simplicity the one-dimensional case, i.e. the delay differential equation $\dot x(t) = b x(t) + c x(t-1)$ with $c > 0$. It is well known that the spectrum of $A$ consists of countably many solutions to the equation $\gamma = b + c e^{-\gamma}$ (see e.g. \cite[Theorem I.4.4]{diekmann1995}), with one real solution $\gamma_*$ which is the dominant eigenvalue. Note that $\lambda = b$ and $\gamma > b$. Let $\mu = \varepsilon + \gamma_*$. Then
\begin{align*}
 (\mu - \lambda)^2 - c^2 e^{-2 \mu} & = (\gamma_* + \varepsilon - b)^2 - c^2 e^{-2 (\gamma_* + \varepsilon)} \\
& = (\gamma_* - b)^2 -c^2 e^{-2 \gamma_*} + 2 \varepsilon(\gamma_* - b + c^2 e^{-2 \gamma_*}) + O(\varepsilon^2) \\
& = 2 \varepsilon( \gamma_* - b + (\gamma_* - b)^2 ) + O(\varepsilon^2) > 0 \quad \mbox{for sufficiently small} \ \varepsilon.
\end{align*}
Hence for this $\mu$, by Theorem~\ref{thm:dissipativityboundedvariation}, we find that $A - \mu I$ is dissipative. This shows that for $\mu$ slightly larger than the spectral bound of $A$, we may already obtain dissipativeness, so Theorem~\ref{thm:dissipativityboundedvariation} is sharp in that sense.

\subsection{The unitary dilation theorem}

An important application of dissipativity is provided by the following fundamental theorem (see \cite{davies1976}, Theorem 7.2.1), which states that we may extend a contractive semigroup on a Hilbert space to a unitary group on a larger Hilbert space.

\begin{theorem}[Szek\H{o}falvi-Nagy's theorem on unitary dilations]
Suppose $(T(t))_{t \geq 0}$ is a strongly continuous contraction semigroup on a Hilbert space $X$. Then there exists a Hilbert space $Z$ such that $Z$ is closed linear subspace of $Z$, and a strongly continuous unitary group $(U(t))_{t \geq 0}$ such that $T(t) = P U(t)$, $t \geq 0$, where $P : Z \rightarrow X$ denotes the linear projection onto $X$.
\end{theorem}

\subsection{Stochastic evolution equations}
Contraction semigroups are useful when establishing existence and uniqueness of invariant measure of stochastic evolution equations.
Consider the following stochastic evolution
\begin{equation} \label{eq:stochasticevolutioneq} \left\{ \begin{array}{ll} d X(t) = \left[ A X(t) + F(X(t)) \right] \ d t + G(X(t)) \ d W(t),\quad & t \geq 0 \\
                          X(0) = x,
                         \end{array} \right. 
\end{equation}
or in mild form
\begin{equation} X(t) = T(t) x + \int_0^t T(t-s) F(X(s)) \ d s + \int_0^t T(t-s) G(X(s)) \ d W(s),
\end{equation}
where $W$ is a cylindrical, mean-zero Wiener process with reproducing kernel Hilbert Space $U$, $A$ is the generator of a strongly continuous semigroup $(T(t))_{t \geq 0}$, $F : X \rightarrow X$ and $G : X \rightarrow L_{\mathrm{HS}}(U;X)$ Lipschitz continuous (where $L_{\mathrm{HS}}(U;X)$ denotes the Banach space of Hilbert-Schmidt operators from $U$ into $X$ (see \cite{dapratozabczyk1996} for definitions).

\subsubsection{Invariant measure}

The following general result on the existence and uniqueness of an invariant measure relies on dissipativity of the semigroup generator.

\begin{theorem}
\label{thm:invariantmeasuredissipativecase}
Suppose there exists $\omega > 0$ such that
\begin{equation} \label{eq:invmeasure_condition} 2 \langle A (x-y) + F(x) - F(y), x - y \rangle + |G(x) - G(y)|_{L_{\mathrm{HS}}(U;X)}^2 \leq - \omega |x-y|^2 \end{equation}
for all $x, y \in X$ and $n \in \N$.

Then there exists exactly one invariant measure $\mu$ for~\eqref{eq:stochasticevolutioneq},
it is strongly mixing and there exists $C > 0$ such that for any bounded Lipschitz continuous function $\varphi : X \rightarrow \R$, all $t > 0$ and $x \in X$,
\[ \left| P(t) \varphi (x) - \int_X \varphi \ d \mu \right| \leq C (1 + |x|) e^{-\omega t / 2} [\varphi]_{\mathrm{Lip}}.\]
\end{theorem}
As usual, for a Lipschitz continuous function $\varphi : X \rightarrow Y$, with $(Y,|\cdot|)$ some Banach space,  $[\varphi]_{\mathrm{Lip}}$ denotes the smallest Lipschitz constant, i.e.
\[ [\varphi]_{\mathrm{Lip}} := \inf\{ c > 0 : |\varphi(x) - \varphi(y)| \leq c |x - y| \mbox{ for all $x, y \in X$} \}.\]

\begin{proof}
See \cite[Theorem 6.3.2]{dapratozabczyk1996}. The proof of that theorem requires~\eqref{eq:invmeasure_condition} to hold for the Yosida approximants but it is easy to see that when $\langle A x, x \rangle \leq \lambda |x|^2$, then for arbitrary $\varepsilon > 0$, $n$ large enough,
\[ \langle A_n x, x \rangle \leq (\lambda + \varepsilon) |x|^2, \]
where $A_n$ denote the Yosida approximants of $A$.
\end{proof}

Combining the above theorem with the main result of this paper, we obtain the following.

\begin{corollary}
Consider in $\R^n$ the stochastic delay differential equation
\begin{equation} \label{eq:sdde} d x(t) = [B x(t) + C x(t-1) + f(x(t)) ] \ d t + g(x(t)) \ d W(t),\end{equation}
with $B, C \in \R^{n \times n}$, $f: \R^n \rightarrow \R^n$ and $g: \R^n \rightarrow \R^{n \times k}$
and such that $B - \lambda$ is dissipative, $f$ and $g$ are globally Lipschitz (where we use the Hilbert-Schmidt norm on $\R^{n \times k}$), and where $W$ is a $k$-dimensional Brownian motion.
Suppose $\omega > 0$ is such that
\[ 2 (\lambda + ||C|| e^{-\lambda} ) + 2 [f]_{\mathrm{Lip}} + [g]_{\mathrm{Lip}}^2 < - \omega.\]

Then there exists exactly one invariant measure on the infinite dimensional state space $\R^n \times L^2([-1,0];\R^n)$ for~\eqref{eq:sdde}.
\end{corollary}

\begin{proof}
This is an immediate result of applying Corollary~\ref{cor:dissipativitydelay}, (i), with $\mu = \lambda + ||C|| e^{-\lambda} + \varepsilon$ for some sufficiently small $\varepsilon > 0$ (which establishes that the corresponding delay semigroup $A$ has the property that $A - \mu$ is dissipative), in conjunction with Theorem~\ref{thm:invariantmeasuredissipativecase}.
\end{proof}

\begin{remark}
The result of this section also applies to systems with L\'evy noise, see \cite[Theorem 16.5]{peszatzabczyk} for the underlying result in this case that is analogous to Theorem~\ref{thm:invariantmeasuredissipativecase}.
\end{remark}

\subsubsection{Stability}
It is also possible to obtain the following stability result using the dissipativity property of a delay semigroup (see \cite{bierkensphdthesis}, Section 6.6.3, and \cite{bierkenslyapunovunbounded}).

\begin{theorem}
Consider the one-dimensional stochastic delay differential equations
\[ d x(t) = [b x(t) + c x(t-1)  ] \ d t + \sigma x(t) \ d W(t),\]
with $b$, $c$ and $\sigma > 0$ such that
$b < \half \sigma^2$ and
\[ |c| < e^{-3/2 \sigma^2} (\half \sigma^2 - b).\]
Then the solution $(x(t))_{t \geq 0}$ is exponentially stable, almost surely.
\end{theorem}

This may be interpreted as follows: if $b < \half \sigma^2$, then the equation without the delay term is almost surely exponentially stable. The second condition then gives a range for $|c|$ so that the system perturbed by an additional delay term remains stable.

\subsection*{Acknowledgements}
O. van Gaans acknowledges the support by a ‘VIDI subsidie’
(639.032.510) of the Netherlands Organisation for Scientific Research (NWO).

We wish to thank the anonymous referee for his valuable comments which have significantly improved the above exposition.


\begin{thebibliography}{10}

\bibitem{batkai}
A.~B{\'a}tkai and S.~Piazzera.
\newblock {\em {Semigroups for Delay Equations}}.
\newblock AK Peters, Ltd., 2005.

\bibitem{bierkensphdthesis}
J.~Bierkens.
\newblock {\em {Long Term Dynamics of Stochastic Evolution Equations}}.
\newblock PhD thesis, Universiteit Leiden, 2009.

\bibitem{bierkenslyapunovunbounded}
J.~Bierkens.
\newblock {Pathwise Stability of Degenerate Stochastic Evolutions}.
\newblock {\em Integral Equations and Operator Theory}, 69 (2011), no. 1, 1–27.

\bibitem{curtainzwart}
R.F. Curtain and H.~Zwart.
\newblock {\em {An Introduction to Infinite-Dimensional Linear Systems
  Theory}}.
\newblock Springer, 1995.

\bibitem{dapratozabczyk1996}
G.~Da~Prato and J.~Zabczyk.
\newblock {\em {Ergodicity for Infinite Dimensional Systems}}.
\newblock Cambridge University Press, 1996.

\bibitem{davies1976}
E.~B. Davies.
\newblock {\em Quantum theory of open systems}.
\newblock Academic Press, London, 1976.

\bibitem{diekmann1995}
O.~Diekmann, S.A. van Gils, S.M. Verduyn~Lunel, and H.O. Walther.
\newblock {\em {Delay Equations: Functional-, Complex-, and Nonlinear
  Analysis}}.
\newblock Springer, 1995.

\bibitem{engelnagel}
K.J. Engel and R.~Nagel.
\newblock {\em {One-Parameter Semigroups for Linear Evolution Equations}}.
\newblock Springer, 2000.

\bibitem{peszatzabczyk}
S.~Peszat and J.~Zabczyk.
\newblock {\em Stochastic partial differential equations with {L}\'evy noise},
  volume 113 of {\em Encyclopedia of Mathematics and its Applications}.
\newblock Cambridge University Press, Cambridge, 2007.

\bibitem{taylor1996}
M.~E. Taylor.
\newblock {\em {Partial Differential Equations: Basic Theory}}.
\newblock Springer, 1996.

\bibitem{webb1976}
G.~F. Webb.
\newblock Functional differential equations and nonlinear semigroups in
  {$L^{p}$}-spaces.
\newblock {\em J. Differential Equations}, 20(1):71--89, 1976.

\end{thebibliography}

\end{document}